\documentclass[11pt]{article}

\usepackage{tocloft}
\title{Notes on Lax Ends}
\author{Kengo Hirata
  \thanks{
    Research Institute for Mathematical Sciences, Kyoto University, Kyoto 606-8502, Japan\newline
    email: {\tt khirata@kurims.kyoto-u.ac.jp}
  }
}
\date{\today}

\usepackage{amsthm, amsfonts}
\usepackage{amssymb}
\usepackage{mathtools}
\usepackage{tikz-cd}
\usepackage{extarrows}
\usepackage{stmaryrd}
\usepackage{xcolor}
\definecolor{darkblue}{rgb}{0.2,0,0.6}
\usepackage[%
  pagebackref=true,%
  colorlinks=true,%
  allcolors=darkblue,%
]{hyperref}
\usepackage{cleveref}
\usepackage{esint}
\usepackage{stackengine,scalerel}
\usepackage{thmtools}
\usepackage{bbm}
\usepackage{subfiles}
\usepackage{ifdraft}
\usepackage{musicography}
\usepackage{todonotes}

\numberwithin{equation}{section}

\declaretheoremstyle[
  spaceabove=6pt, spacebelow=6pt,%
  headfont=\normalfont\itshape,%
  notefont=\mdseries,%
  notebraces={(}{)},%
  bodyfont=\normalfont,%
  postheadspace=0.5em,%
  qed=$\square$%
]{myproofstyle}

\declaretheorem[style=plain,numberwithin=section,name=Theorem]{theorem}
\declaretheorem[style=plain,sibling=theorem,name=Lemma]{lemma}
\declaretheorem[style=plain,sibling=theorem,name=Proposition]{proposition}
\declaretheorem[style=plain,sibling=theorem,name=Corollary]{corollary}

\declaretheorem[style=definition,qed=$\blacksquare$,sibling=theorem,name=Definition]{definition}
\declaretheorem[style=definition,qed=$\blacksquare$,sibling=theorem,name=Example]{example}

\usepackage{tikz}
\usetikzlibrary{arrows.meta,decorations.pathreplacing,decorations.markings,shapes,calc}
\usetikzlibrary{matrix,arrows,decorations.pathmorphing,backgrounds,decorations.markings,positioning}

\tikzset{1cell/.style={->, labelsize, auto}}

\tikzset{2cell/.style={-implies,double,double equal sign distance,shorten 
>=2pt, shorten <=3pt}}
\tikzset{2cellshort/.style={-implies,double,double equal sign distance,shorten 
>=4pt, shorten <=5pt}}
\tikzset{2cellr/.style={implies-,double,double equal sign distance,shorten 
>=3pt, shorten <=2pt}}

\tikzset{3cell/.style={-implies,double,double distance=2.5pt,shorten >=2pt, 
shorten <=3pt}}
\tikzset{labelsize/.style={font=\scriptsize}}
\tikzset{string/.style={very thick}}
\tikzset{
  pto/.style={->,postaction={decorate},
    decoration={
        markings,
        mark=at position 0.5 with {\arrow{|}}}
  },
}
\newcommand{\twocell}[5][]{\draw[2cell] (#2)++(#4) to node[auto,labelsize,#1]{#3} ++(#5);}

\newcommand{\twocelll}[3][]{\twocell[#1]{#2}{#3}{0.25,0}{-0.5,0}}
\newcommand{\twocellr}[3][]{\twocell[#1]{#2}{#3}{-0.25,0}{0.5,0}}
\newcommand{\twocelld}[3][]{\twocell[#1]{#2}{#3}{0,0.25}{0,-0.5}}

\newcommand{\twocellul}[3][]{\twocell[#1]{#2}{#3}{0.175,-0.175}{-0.35,0.35}}

\newcommand{\twocelldl}[3][]{\twocell[#1]{#2}{#3}{0.175,0.175}{-0.35,-0.35}}

\crefname{equation}{}{}

\newcommand{\id}{\mathrm{id}}

\newcommand{\op}{\mathrm{op}}
\newcommand{\co}{\mathrm{co}}
\newcommand{\ob}{\mathrm{ob}}

\newcommand{\Hom}{\mathrm{Hom}}

\newcommand{\laxlim}{{\mathrm{lax\,lim}}}
\newcommand{\laxcolim}{\mathrm{lax\,colim}}
\newcommand{\oplaxlim}{{\mathrm{oplax\,lim}}}

\newcommand{\iso}{\cong}
\newcommand{\eqv}{\simeq}

\newcommand{\power}{\pitchfork}
\newcommand{\copower}{\bullet}
\newcommand{\dflat}{{\musDoubleFlat}}
\newcommand{\dsharp}{{\musDoubleSharp}}

\newcommand{\cat}[1]{\mathcal{#1}}
\newcommand{\Lax}[1]{{\mathbf{Lax}[#1]}}
\newcommand{\Ps}[1]{{\mathbf{Ps}[#1]}}
\newcommand{\wtilde}[1]{\widetilde{#1}}

\newcommand{\wlim}[2]{{\lim\nolimits^{#1} #2}}
\newcommand{\wcolim}[2]{{\mathrm{colim}^{#1} #2}}

\newcommand{\Cat}{\mathbf{Cat}}
\newcommand{\A}{\cat{A}}
\newcommand{\B}{\cat{B}}

\newcommand{\K}{\cat{K}}
\newcommand{\One}{\mathbf{1}}
\newcommand{\one}{\mathbbm{1}}
\newcommand{\two}{\mathbbm{2}}

\NewDocumentCommand{\lend}{e{_^}}{
  \newintop{\hspace{-0.09em}\leftarrow}[#1][#2]%
}%
\NewDocumentCommand{\oplend}{e{_^}}{
  \newintop{\hspace{-0.09em}\rightarrow}[#1][#2]%
}%
\NewDocumentCommand{\psend}{e{_^}}{
  \newintop{\sim}[#1][#2]%
}%
\NewDocumentCommand{\newintop}{moo}{%
  \ThisStyle{\ensurestackMath{%
  \stackengine{0pt}{%
    \SavedStyle\int%
      \IfValueT{#2}{_{#2}}%
      \IfValueT{#3}{^{#3}}%
  }{%
    \SavedStyle#1%
  }{O}{l}{F}{F}{L}}}}%

\begin{document}
\maketitle
\begin{abstract}
  In enriched category theory, the notion of extranatural transformations is more fundamental than that of ordinary natural transformations, and the ends, the universal extranatural transformations, play a critical role.
  On the other hand, 2-category theory makes use of several other natural transformations, such as lax and pseudo transformations.
  For these weak transformations, it is known that we can define the corresponding extranatural transformations or ends.
  However, there is little literature describing such results in detail.
  We provide a detailed calculation of the lax end, including its relation to the lax limits.
  We prove the bicategorical coYoneda lemma as the dual of the bicategorical Yoneda lemma, and also show that the weight of any lax end is a PIE weight, but it might not be a weight for a lax limit.
\end{abstract}
\tableofcontents

\section{Introduction}

\subsection{Background}

While 2-categories can be defined as $\Cat$-enriched categories, there are various types of weak notions of functors and weak transformations. These weak gadgets cannot be defined in general enriched categories, but it is known that for many statements in enriched category theory, it is also possible to prove statements that are replaced by weak gadgets in 2-category theory. However, while the proofs of those statements in enriched category theory can be written concisely using ends, the corresponding theorems with weak gadgets in 2-category theories are written using jumbled diagrams and are not treated in a very unified manner.

Bozapalides introduced \emph{lax end} as a universal \emph{extraordinary lax natural transformation}, just as an end is a universal extraordinary natural transformation \cite{Bozapalides1977,Bozapalides1980}. However, the notion of lax ends has not been well studied, and the only reference containing a survey on lax ends is \cite{loregian_2021}. In these notes, we give detailed calculations and proofs of basic results on lax ends in 2-category theory. In particular, we treat some calculations on weighted lax limits in \Cref{section:lax_limits_4} and the bicategorical (co)Yoneda lemma in \Cref{section:yoneda_7}.
In \Cref{section:pie_lend_llim_6}, we also examine the class of limits where lax ends live.

For simplicity, in these notes, we restrict our attention to 2-categories and 2-functors, although Bozapalide defined lax ends for lax (or quasi) functors between 2-categories in \cite{Bozapalides1977,Bozapalides1980} and  Corner proved coYoneda lemma for bicategories and pseudo functors in \cite{Alexander2016coyoneda}.

\subsection{Notations}
Each term 2-categories, 2-functors, and 2-transformations implies $\Cat$-enriched categories, functors, and transformations.
Other morphisms of 2-categories include (op)lax and pseudo functors and the same for transformations.
Out of these kinds of functors, we only deal with 2-functors in these notes.
The functor categories which appear in these notes will be denoted as follows.
\begin{itemize}
  \item $\Lax{\A, \B}$: 2-functors, lax transformations, and modifications
  \item $\Ps{\A, \B}$: 2-functors, pseudo transformations, and modifications
  \item $[\A, \B]$: 2-functors, 2-transformations, and modifications
\end{itemize}

If $F\colon \A \rightarrow \Cat$ and $G\colon \A \rightarrow \K$ are 2-functors, the $F$ weighted limit of the diagram $G$ is denoted as $\wlim{F}{G}$.
Similarly, $\wcolim{F}{G}$ is the $F$ weighted colimit of the diagram $G$.
Weighted lax limits, which will be defined in \Cref{section:lax_limits_4}, will be denoted as $\laxlim^{F}{G}$.

As special weighted limits, we have ends and powers.
We write the end of $T \colon \A\times\B\rightarrow\K$ as $\int_{A\in\A}T(A,A)$, and the coend as $\int^{A\in\A} T(A,A)$.
If $K\in\K$ and $\mathbb{A} \in \Cat$, we write the power as $\mathbb{A}\power K$ and the copower as $\mathbb{A}\copower K$.

To denote 2-categories, we tend to use the calligraphic font $\A,\B,\K,\dots$,
and for 1-categories, the blackboard bold fonts $\mathbb{A},\one,\dots$.
For the terminal 2-category, we use $\One$.
The enriched yoneda embedding functor is denoted as $Y\colon\A \rightarrow [\A^\op, \Cat]$.

Let $G\colon\A\rightarrow\B$ a 2-functor.
We define two 2-functors $\widehat{G}\colon \B^\op\rightarrow[\A,\Cat]$ and $\widetilde{G}\colon\B\rightarrow[\A^\op,\Cat]$ by $\widehat{G} = \B(-,G?)$ and $\widetilde{G} = \B(G?,-)$.

\section{Lax naturalities}

\begin{definition}
  Let $T\colon \A^\op\times\A\rightarrow\K$ be a 2-functor and $K$ an object of $\K$.
  A \emph{lax wedge} (or \emph{extraordinary lax natural transformation})
  $\sigma\colon K\rightarrow T$ consists of:
  \begin{itemize}
    \item a 1-cell $\sigma_A\colon K\rightarrow T(A, A)$ for each object $A$ in $\A$;
    \item a 2-cell $\sigma_f\colon T(1_A,f)\sigma_A \rightarrow T(f,1_B)\sigma_B$
          for each 1-cell $f\colon A \rightarrow B$
          as shown in the diagram below;
          \[
            \begin{tikzpicture}
              \node(K) at (0,0) {$K$};
              \node(A) at (3,0) {$T(A,A)$};
              \node(B) at (0,-2) {$T(B,B)$};
              \node(AB) at (3,-2) {$T(A,B)$};
              \draw [1cell] (K) to node {$\sigma_A$} (A);
              \draw [1cell] (K) to node[swap] {$\sigma_B$} (B);
              \draw [1cell] (A) to node {$T(1_A,f)$} (AB);
              \draw [1cell] (B) to node[swap] {$T(f,1_B)$} (AB);
              \twocelldl[swap]{1.5,-1}{$\sigma_f$}
            \end{tikzpicture}
          \]
          satisfying the following three equalities:
          \begin{align}
            \begin{tikzpicture}[x=9.5mm,y=9mm,baseline=(A.base)]
              \node(K) at (0,0) {$K$};
              \node(A) at (2,-2) {$T(A,A)$};
              \node(B) at (-2,-2) {$T(B,B)$};
              \node(AB) at (0,-4) {$T(A,B)$};
              \draw [1cell] (K) to node {$\sigma_A$} (A);
              \draw [1cell] (K) to node[swap] {$\sigma_B$} (B);
              \draw [1cell, bend left] (A) to node {$T(1,f)$} (AB);
              \draw [1cell, bend right, near end, swap] (A) to node {$T(1,g)$} (AB);
              \draw [1cell, bend right] (B) to node[swap] {$T(g,1)$} (AB);
              \twocelll[swap]{0,-1.9}{$\sigma_g$}
              \twocellul[swap, near end]{0.7,-3.2}{$T(1,\alpha)$}
            \end{tikzpicture}
            \quad
             & =
            \quad
            \begin{tikzpicture}[x=9.5mm,y=9mm,baseline=(A.base)]
              \node(K) at (0,0) {$K$};
              \node(A) at (2,-2) {$T(A,A)$};
              \node(B) at (-2,-2) {$T(B,B)$};
              \node(AB) at (0,-4) {$T(A,B)$};
              \draw [1cell] (K) to node {$\sigma_A$} (A);
              \draw [1cell] (K) to node[swap] {$\sigma_B$} (B);
              \draw [1cell, bend left] (A) to node {$T(1,f)$} (AB);
              \draw [1cell, bend left, near end] (B) to node {$T(f,1)$} (AB);
              \draw [1cell, bend right] (B) to node[swap] {$T(g,1)$} (AB);
              \twocelll[swap]{0,-1.9}{$\sigma_f$}
              \twocelldl[near start, swap]{-0.7,-3.2}{$T(1,\alpha)$}
            \end{tikzpicture}
            \label{eq:wedge_naturality} \\
            \begin{tikzpicture}[baseline=(m.base)]
              \node(K) at (0,0) {$K$};
              \node(A) at (3,0) {$T(A,A)$};
              \node(B) at (0,-2) {$T(A,A)$};
              \node(AB) at (3,-2) {$T(A,A)$};
              \draw [1cell] (K) to node {$\sigma_A$} (A);
              \draw [1cell] (K) to node(m)[swap] {$\sigma_A$} (B);
              \draw [1cell] (A) to node {$1$} (AB);
              \draw [1cell] (B) to node[swap] {$1$} (AB);
              \twocelldl[swap]{1.5,-1}{$\sigma_{1_{A}}$}
            \end{tikzpicture}
            \quad
             & =
            \quad
            \mathrm{identity}
            \label{eq:wedge_identity}   \\
            \begin{tikzpicture}[x=11mm,y=10mm,baseline=(K.base)]
              \node(K) at (0,0) {$K$};
              \node(A) at (1,1.73) {$T(A,A)$};
              \node(C) at (1,-1.73) {$T(C,C)$};
              \node(AB) at (3,1.73) {$T(A,B)$};
              \node(BC) at (3,-1.73) {$T(B,C)$};
              \node(AC) at (4,0) {$T(A,C)$};
              \node(B) at (2,0) {$T(B,B)$};
              \draw [1cell] (K) to node {$\sigma_A$} (A);
              \draw [1cell] (K) to node[swap] {$\sigma_C$} (C);
              \draw [1cell] (A) to node {$T(1,f)$} (AB);
              \draw [1cell] (C) to node[swap] {$T(g,1)$} (BC);
              \draw [1cell] (AB) to node {$T(1,g)$} (AC);
              \draw [1cell] (BC) to node[swap] {$T(f,1)$} (AC);
              \draw [1cell] (K) to node[swap] {$\sigma_B$} (B);
              \draw [1cell] (B) to node[swap] {$T(f,1)$} (AB);
              \draw [1cell] (B) to node[] {$T(1,g)$} (BC);
              \twocelld[swap]{1.7,0.87}{$\sigma_f$}
              \twocelld[swap]{1.7,-0.87}{$\sigma_g$}
            \end{tikzpicture}
            \quad
             & =
            \quad
            \begin{tikzpicture}[x=11mm,y=10mm,baseline=(K.base)]
              \node(K) at (0,0) {$K$};
              \node(A) at (1,1.73) {$T(A,A)$};
              \node(C) at (1,-1.73) {$T(C,C)$};
              \node(AB) at (3,1.73) {$T(A,B)$};
              \node(BC) at (3,-1.73) {$T(B,C)$};
              \node(AC) at (4,0) {$T(A,C)$};
              \draw [1cell] (K) to node {$\sigma_A$} (A);
              \draw [1cell] (K) to node[swap] {$\sigma_C$} (C);
              \draw [1cell] (A) to node {$T(1,f)$} (AB);
              \draw [1cell] (C) to node[swap] {$T(g,1)$} (BC);
              \draw [1cell] (AB) to node {$T(1,g)$} (AC);
              \draw [1cell] (BC) to node[swap] {$T(f,1)$} (AC);
              \draw[1cell, sloped, swap] (A) to node {$T(1,gf)$} (AC);
              \draw[1cell, sloped] (C) to node {$T(gf,1)$} (AC);
              \twocelld[swap]{1.6,-0}{$\sigma_{gf}$}
            \end{tikzpicture}
            \label{eq:wedge_composition}
          \end{align}
  \end{itemize}
  Dually, a \emph{lax cowedge} is a lax wedge in $\K^\op$.
\end{definition}

Of course, an \emph{oplax wedge} $K\rightarrow T$ can also be defined as a lax wedge in $\K^\co$,
but this is the same as a lax wedge $K\rightarrow T'$ with $T'(A,B) = T(B,A)$.
Also, we will briefly discuss \emph{pseudo wedges} later in \Cref{section:yoneda_7}.

\begin{definition}
  Let $\sigma, \tau\colon K\rightarrow T$ be a pair of lax wedges.
  A \emph{modification} $\Gamma$ from $\sigma$ to $\tau$ is a family of 2-cells
  $\left\{\Gamma_A\colon \sigma_A \rightarrow\tau_A\right\}_{A\in\A}$, satisfying
  \[
    \begin{tikzpicture}[x=9.3mm,y=9mm,baseline=(A.base)]
      \node(K) at (0,0) {$K$};
      \node(A) at (2,-2) {$T(A,A)$};
      \node(B) at (-2,-2) {$T(B,B)$};
      \node(AB) at (0,-4) {$T(A,B)$};
      \draw [1cell, bend left=35] (K) to node {$\sigma_A$} (A);
      \draw [1cell, bend left=35, near start] (K) to node {$\sigma_B$} (B);
      \draw [1cell, bend right=35] (K) to node[swap] {$\tau_B$} (B);
      \draw [1cell] (A) to node {$T(1,f)$} (AB);
      \draw [1cell] (B) to node[swap] {$T(f,1)$} (AB);
      \twocelll[swap]{0,-2.3}{$\sigma_f$}
      \twocelll[swap]{-0.9,-1.1}{$\Gamma_f$}
    \end{tikzpicture}
    \quad
    =
    \quad
    \begin{tikzpicture}[x=9.3mm,y=9mm,baseline=(A.base)]
      \node(K) at (0,0) {$K$};
      \node(A) at (2,-2) {$T(A,A)$};
      \node(B) at (-2,-2) {$T(B,B)$};
      \node(AB) at (0,-4) {$T(A,B)$};
      \draw [1cell, bend left=35] (K) to node {$\sigma_A$} (A);
      \draw [1cell, bend right=35, near start, swap] (K) to node {$\tau_A$} (A);
      \draw [1cell, bend right=35] (K) to node[swap] {$\tau_B$} (B);
      \draw [1cell] (A) to node {$T(1,f)$} (AB);
      \draw [1cell] (B) to node[swap] {$T(f,1)$} (AB);
      \twocelll[swap]{0,-2.3}{$\tau_f$}
      \twocelll[swap]{0.9,-1.1}{$\Gamma_f$}
    \end{tikzpicture}.
  \]
\end{definition}

To see that this definition is reasonable, we first check that the lax wedges are compatible with 
the usual lax transformations of 2-functors.
That is, we check that a lax transformation $F\rightarrow G$ can be regarded as
a lax wedge of type $\one\rightarrow \K(F?,G-)$.

\begin{proposition}\label{thm:wedge_and_transformation}
  Let $F, G \colon \A^\op \rightarrow \K$ be 2-functors.
  Lax transformations $\alpha\colon F \rightarrow G$
  are bijective to lax wedges of type $\one \rightarrow \K(F?, G-)$.
  And also, modifications between lax transformations are
  bijective to modifications of corresponding lax wedges.
\end{proposition}
\begin{proof}
  One can easily check that a family $\left\{\sigma_A\colon\one\rightarrow\K(FA,GA) \right\}_{A\in\A}$
  corresponds to a family $\left\{\bar\sigma_A\colon FA\rightarrow GA \right\}_{A\in\A}$,
  and $\{\sigma_f\colon\K(1,f)\sigma_A\rightarrow\K(f,1)\sigma_B\}$
  to $\{\bar\sigma_f\colon Gf\circ\bar\sigma_A\rightarrow\bar\sigma_B\circ Ff\}$.
  Also, it can be checked that the coherence required for lax wedges and for lax transformations are equivalent.
\end{proof}

Unlike general enriched categories, lax naturality in $(A,B) \in \A\times\B$
can not be simply verified by checking the lax naturality in $A$ for each fixed $B$ and vice versa.
It requires additional equalities.

\begin{proposition}\label{prop:lax_naturality_for_each_variable}
  Let $T\colon\A^\op\times\B^\op\times\A\times\B\rightarrow\K$ be a 2-functor,
  and assume we have
  \begin{itemize}
    \item a family of 1-cells $\{\sigma_{AB}\colon K\rightarrow T(A,B,A,B)\}$ in $\K$,
    \item families of 2-cells ${\{\sigma_{fB}\colon T(1,1,f,1)\sigma_{AB} \rightarrow T(f,1,1,1)\sigma_{A'B}\}}_f$ for each $B$, and
    \item families of 2-cells ${\{\sigma_{Ag}\colon T(1,1,1,g)\sigma_{AB} \rightarrow T(1,g,1,1)\sigma_{AB'}\}}_g$ for each $A$.
  \end{itemize}
  Then, $\sigma_{-B}$ and $\sigma_{A-}$ are wedges for each $B$ and $A$ satisfying the following equality,
  \begin{align*}
    \begin{tikzpicture}[x=12mm,y=9mm,baseline=(AC.base)]
      \node(K) at (0,0) {$K$};
      \node(A) at (1,1.73)   {$TABAB$};
      \node(C) at (1,-1.73)  {$TA'B'A'B'$};
      \node(AB) at (3,1.73)  {$TABA'B$};
      \node(BC) at (3,-1.73) {$TA'BA'B'$};
      \node(AC) at (4,0)     {$TABA'B'$};
      \node(B) at (2,0)      {$TA'BA'B$};
      \draw [1cell] (K) to node {$\sigma_{AB}$} (A);
      \draw [1cell] (K) to node[swap] {$\sigma_{A'\!B'}$} (C);
      \draw [1cell] (K) to node[swap] {$\sigma_{A'\!B}$} (B);
      \draw [1cell] (A) to node        {$T(11f1)$} (AB);
      \draw [1cell] (C) to node[swap]  {$T(1g11)$} (BC);
      \draw [1cell] (AB) to node       {$T(111g)$} (AC);
      \draw [1cell] (BC) to node[swap] {$T(f111)$} (AC);
      \draw [1cell] (B) to node[swap]  {$T(f111)$} (AB);
      \draw [1cell] (B) to node[]      {$T(111g)$} (BC);
      \twocelld[swap]{1.7,0.87}{$\sigma_{fB}$}
      \twocelld[swap]{1.7,-0.87}{$\sigma_{A'g}$}
    \end{tikzpicture}
    \ =\
    \begin{tikzpicture}[x=12mm,y=9mm,baseline=(AC.base)]
      \node(K) at (0,0) {$K$};
      \node(A) at (1,1.73)   {$TABAB$};
      \node(B) at (2,0)      {$TAB'AB'$};
      \node(C) at (1,-1.73)  {$TA'B'A'B'$};
      \node(AB) at (3,1.73)  {$TABAB'$};
      \node(BC) at (3,-1.73) {$TAB'A'B'$};
      \node(AC) at (4,0)     {$TABA'B'$};
      \draw [1cell] (K) to node       {$\sigma_{AB}$} (A);
      \draw [1cell] (K) to node[swap] {$\sigma_{A'B'}$} (C);
      \draw [1cell] (K) to node[swap] {$\sigma_{AB'}$} (B);
      \draw [1cell] (A) to node        {$T(111g)$} (AB);
      \draw [1cell] (C) to node[swap]  {$T(f111)$} (BC);
      \draw [1cell] (AB) to node       {$T(11f1)$} (AC);
      \draw [1cell] (BC) to node[swap] {$T(1g11)$} (AC);
      \draw [1cell] (B) to node[swap]  {$T(1g11)$} (AB);
      \draw [1cell] (B) to node[]      {$T(11f1)$} (BC);
      \twocelld[swap]{1.7,0.87}{$\sigma_{Ag}$}
      \twocelld[swap]{1.7,-0.87}{$\sigma_{fB'}$}
    \end{tikzpicture}
  \end{align*}
  if and only if $\sigma$ is a wedge
  for $T\colon{(\A\times\B)}^\op\times(\A\times\B)\rightarrow\K$ with $\sigma_{fg}$ defined by this 2-cell.
\end{proposition}
\begin{proof}
  The ``if'' part is trivial and the other part can be shown by pasting some diagrams.
\end{proof}

\label{section:lax_naturality_2}

\section{Lax ends in Cat}

As in enriched category theory, we wish to define the lax end as a universal lax wedge.
In this section, we define a lax end for a $\Cat$-valued 2-functor $\A^\op\times\A\rightarrow\Cat$,
and examine its properties in this and the following two sections.
Of course, we can also define a lax end for general 2-category $\K$, which will be discussed later in \Cref{section:general_lax_end_5}.

\begin{definition}
  A \emph{lax end} of a 2-functor $T\colon\A^\op \times \A \rightarrow \Cat$ consists of a category
  $\lend_A T(A,A)$ and a lax wedge $\lambda\colon\lend_A T(A,A) \rightarrow T$,
  with the following universal properties:
  \begin{description}
    \item[1-dimensional]
      For each lax wedge $\sigma\colon X\rightarrow T$, there is a unique functor
      $u\colon X \rightarrow \lend_AT(A,A)$ satisfying $\sigma = \lambda u$, that is,
      $\sigma_A = \lambda_A u$ for each object in $\A$,
      and $\sigma_f = \lambda_fu$ for each 1-cell in $\A$.
    \item[2-dimensional]
      For each modification $\Gamma\colon \lambda u\rightarrow\lambda v\colon X \rightarrow T$,
      There is a unique 2-cell $\gamma \colon u\rightarrow v$
      such that $\Gamma=\lambda*\gamma$, that is, $\Gamma_A = \lambda_A\gamma$
      for each $A$ in $\A$.
  \end{description}

  Dually, we also define a \emph{lax coend} $T\rightarrow\lend^A T(A,A)$ as a universal cowedge in the same way.
\end{definition}

We show that there is a lax (co)end for every $T\colon \A^\op\times\A\rightarrow\Cat$
later in \Cref{prop:existence_in_cat}, but before we prove it,
let us now examine some of the properties that hold when it exists.

Let $T$ be a 2-functor $T\colon\A^\op\times\A\rightarrow\Cat$, and assume that its lax end $\lend_A T(A,A)$ exists.
Then, since the objects of $\lend_AT(A,A)$ correspond to lax wedges of type $\one \rightarrow T$,
its data can explicitly be written down; an object of $\lend_AT(A,A)$ consists of a pair of families
${\{x_A \in T(A,A)\}}_{A\in\A}$ and\\${\{x_f\colon T(1,f)x_A \rightarrow T(f,1)x_B\in T(A,B)\}}_{f\colon A\rightarrow B}$,
  satisfying the three axioms:
  \begin{itemize}
    \item $x_{1_A} = \mathrm{id}$,
    \item for each composable pair $(g, f)$ in $\A$, $T(f,1)x_g\circ T(1,g)x_f = x_{gf}$,
    \item for each 2-cell $\alpha\colon f\rightarrow g$,
          \[
            \begin{tikzcd}[column sep=large]
              {T(1,f)x_A} \rar["{T(1,\alpha)}_{x_A}"] \dar["{x_f}"'] & {T(1,g)x_A} \dar["{x_g}"]\\
              {T(f,1)x_B} \rar["{T(\alpha,1)}_{x_B}"] & {T(g,1)x_B.}
            \end{tikzcd}
          \]
  \end{itemize}
  A morphism from $x$ to $y$ in $\lend_AT(A,A)$ is a family ${\{\gamma_A\colon x_A\rightarrow y_A\}}_A$
  satisfying $T(f,1)\gamma_B \circ x_f = y_f \circ T(1,f)\gamma_A$ for each $f\colon A\rightarrow B$ in $\A$.

  If $\Delta\mathbb{B}\colon \mathbb{A}^\op \times \mathbb{A}\rightarrow \Cat$ is a constant functor which returns a category $\mathbb{B}$, then the lax end $\lend_A \Delta \mathbb{B}$ is the functor category $[\mathbb{A},\mathbb{B}]$.

  Although we gave an explicit presentation of a lax end as above,
  we will not use this presentation below,
  and all the following discussions will be based on its universality.

  First, we show that the two basic propositions -- analogous to the enriched case -- hold in our lax setting.

  \begin{proposition}\label{prop:hom_of_functor_category_as_lax_end}
    There is an isomorphism of categories
    \begin{align}
      \Lax{\A,\K}(F,G) \iso \lend_{A\in\A} \K(FA, GA).
      \label{eq:hom_of_functor_category_as_lax_end}
    \end{align}
  \end{proposition}
  \begin{proof}
    Trivial from \Cref{thm:wedge_and_transformation}.
  \end{proof}

  \begin{proposition}\label{prop:cat_lax_end_commute_representable}
    A lax end commutes with representables
    \begin{align}
      \left[X, \lend_{A\in\A} T(A,A)\right] \iso \lend_{A\in\A}[X, T(A,A)].
      \label{eq:cat_lax_end_commute_representable}
    \end{align}
  \end{proposition}
  \begin{proof}
    From the universal property of lax ends, the objects in left-hand side
    are bijective to lax wedges of form $X \rightarrow T$, and the morphisms
    corresponds to those modifications.
    On the other hand, objects in the right-hand side are
    lax wedges of form $\one \rightarrow [X, T(?,-)]$, and the morphisms are
    those modifications.
    One can easily show these two categories are isomorphic.
  \end{proof}

  In order to show the existence of lax ends,
  we would like to represent a lax end with a weighted limit in $\Cat$
  and deduce the existence by the completeness of $\Cat$.
To this end, we introduce \emph{lax descent objects}.

\begin{definition}
  Coherence data in a 2-category $\K$ consists of
  three objects $X_1$, $X_2$, $X_3$, six 1-cells
  \[
    \begin{tikzpicture}
      \node(dummy) at (0,0.5) {};
      \node(X1) at (3,0) {$X_1$};
      \node(X2) at (0,0) {$X_2$};
      \node(X3) at (-3,0) {$X_3$};
      \draw[1cell, transform canvas={yshift=12pt}, swap]
      (X2) -- node {$r$} (X3);
      \draw[1cell, swap]
      (X2) -- node {$s$} (X3);
      \draw[1cell, transform canvas={yshift=-12pt}, swap]
      (X2) -- node {$t$} (X3);
      \draw[1cell, transform canvas={yshift=12pt}, swap]
      (X1) -- node(l) {$v$} (X2);
      \draw[1cell]
      (X2) -- node {$i$} (X1);
      \draw[1cell, transform canvas={yshift=-12pt}, swap]
      (X1) -- node {$w$} (X2);
    \end{tikzpicture}
  \]
  and five equalities
  \begin{align*}
    \delta & \colon iv = 1 ,
           & \gamma          & \colon 1  = iw, \\
    \kappa & \colon rv = sv,
           & \lambda         & \colon tw = sw, \\
    \rho   & \colon rw = tv.
  \end{align*}
  If we regard the simplex category as a discrete 2-category, this coherence data is a full sub 2-category of the simplex category.
\end{definition}

An example of coherence data (with non-trivial 2-cells) can be found in 2-monad theory \cite{Lack2005codescent}.

\begin{definition}
  A lax descent object of the coherence data consists of
  an object $X$, a 1-cell $x\colon X\rightarrow X_1$, and
  a 2-cell $\xi\colon vx\Rightarrow wx$
  satisfying the following diagram equalities,
  \begin{align}
    \begin{tikzpicture}[baseline=(X.base)]
      \node(X) at (0,0) {$X$};
      \node(U) at (1,1) {$X_1$};
      \node(D) at (1,-1) {$X_1$};
      \node(M) at (2,0) {$X_2$};
      \node(R) at (3.5,0) {$X_1$};
      \draw[1cell] (X) to node {$x$} (U);
      \draw[1cell,swap] (X) to node {$x$} (D);
      \draw[1cell] (U) to node {$v$} (M);
      \draw[1cell,swap] (D) to node {$w$} (M);
      \draw[1cell] (M) to node {$i$} (R);
      \twocelld{1,0}{$\xi$}
    \end{tikzpicture}
    \quad & =\quad
    \begin{tikzpicture}[baseline=(X.base)]
      \node(X) at (-0.5,0) {$X$};
      \node(M) at (1,0) {$X_1$};
      \node(U) at (2,1) {$X_2$};
      \node(D) at (2,-1) {$X_2$};
      \node(R) at (3,0) {$X_1$};
      \draw[1cell] (X) to node {$x$} (M);
      \draw[1cell] (M) to node {$v$} (U);
      \draw[1cell,swap] (M) to node {$w$} (D);
      \draw[1cell] (U) to node {$i$} (R);
      \draw[1cell,swap] (D) to node {$i$} (R);
      \draw[1cell,] (M) to node {$1$} (R);
    \end{tikzpicture}
    \label{eq:lax_descent_small} \\
    \begin{tikzpicture}[baseline=(l.base)]
      \node(Top) at (0,2) {$X_1$};
      \node(LU) at (-1.7,1) {$X$};
      \node(RU) at (1.7,1) {$X_2$};
      \node(MU) at (0,0) {$X_1$};
      \node(Bot) at (0,-2) {$X_2$};
      \node(LD) at (-1.7,-1) {$X_1$};
      \node(RD) at (1.7,-1) {$X_3$};
      \draw[1cell] (LU) to node {$x$} (Top);
      \draw[1cell,swap] (LU) to node {$x$} (MU);
      \draw[1cell,swap] (LU) to node {$x$} (LD);
      \draw[1cell] (Top) to node {$v$} (RU);
      \draw[1cell,swap] (MU) to node {$w$} (RU);
      \draw[1cell] (MU) to node {$v$} (Bot);
      \draw[1cell,swap] (LD) to node {$w$} (Bot);
      \draw[1cell] (RU) to node(l) {$r$} (RD);
      \draw[1cell,swap] (Bot) to node {$t$} (RD);
      \twocelld{0,1}{$\xi$}
      \twocelld{-0.9,-0.5}{$\xi$}
    \end{tikzpicture}
    \quad & =\quad
    \begin{tikzpicture}[baseline=(l.base)]
      \node(Top) at (0,2) {$X_1$};
      \node(LU) at (-1.7,1) {$X$};
      \node(RU) at (1.7,1) {$X_2$};
      \node(Bot) at (0,-2) {$X_2$};
      \node(LD) at (-1.7,-1) {$X_1$};
      \node(RD) at (1.7,-1) {$X_3$};
      \node(MD) at (0,-0) {$X_2$};
      \draw[1cell] (LU) to node {$x$} (Top);
      \draw[1cell,swap] (LU) to node {$x$} (LD);
      \draw[1cell] (Top) to node {$v$} (RU);
      \draw[1cell,swap] (LD) to node {$w$} (Bot);
      \draw[1cell,swap] (LD) to node {$w$} (MD);
      \draw[1cell] (RU) to node(l) {$r$} (RD);
      \draw[1cell,swap] (Bot) to node {$t$} (RD);
      \draw[1cell,swap] (MD) to node {$s$} (RD);
      \draw[1cell] (Top) to node {$v$} (MD);
      \twocelld{-0.9,0.5}{$\xi$}
    \end{tikzpicture}
    \label{eq:lax_descent_large}
  \end{align}
  with the one and two-dimensional universal property.
  The one dimensional universal property is that, if there is another triple
  $(Y,y\colon Y\rightarrow X_1,\eta\colon vy\Rightarrow wy)$ satisfying
  the same equality of 2-cells, there uniquely exists a 1-cell $u\colon Y\rightarrow X$
  such that $xu = y$ and $\xi u = \eta$.
  And the two-dimensional universal property is that, if there is a
  pair of 1-cells $u,u'\colon Y\rightarrow X$ and a 2-cell
  $\alpha\colon xu \Rightarrow xu'$ satisfying $\xi u' \circ v \alpha = w \alpha \circ \xi u$,
  there uniquely exists a 2-cell $\beta\colon u\Rightarrow u'$
  such that $\alpha = x\beta$.
\end{definition}

This lax descent object can easily be obtained by taking several limits.

\begin{lemma}\label{lem:codescent_objects_are_iterated_limits}
  Let $\K$ be a 2-category.
  If $\K$ admits inserters and equifiers,
  then $\K$ also admits lax descent objects.
\end{lemma}
\begin{proof}
  Take an inserter of $v$ and $w$,
  \[
    \begin{tikzpicture}[x=10mm,y=8mm]
      \node(I) at (0,2) {$I$};
      \node(R) at (2,2) {$X_1$};
      \node(L) at (0,0) {$X_1$};
      \node(T) at (2,0) {$X_2$};
      \draw[1cell] (I) to node {$a$} (R);
      \draw[1cell,swap] (I) to node {$a$} (L);
      \draw[1cell] (R) to node {$v$} (T);
      \draw[1cell,swap] (L) to node {$w$} (T);
      \twocelldl{1,1}{$\zeta$}
    \end{tikzpicture}
  \]
  and the lax descent object is obtained by taking equifiers twice for the diagrams in \cref{eq:lax_descent_large,eq:lax_descent_small}
  with $\xi$ properly substituted using $\zeta$.
\end{proof}

Finally, we show that a lax end is a kind of lax descent object,
and prove its existence.

\begin{proposition}\label{prop:existence_in_cat}
  For a small $\A$, any 2-functor $T\colon\A^\op \times \A \rightarrow \Cat$
  has a lax end and a lax coend.
\end{proposition}
\begin{proof}
  Let $X_1$, $X_2$, $X_3$ be as follows.
  \begin{align*}
    X_1 & = \prod_{A\in\A} T(A,A)                             \\
    X_2 & = \prod_{A,B\in\A} [\A(A,B), T(A,B)]                \\
    X_3 & = \prod_{A,B,C\in\A} [\A(B,C)\times\A(A,B), T(A,C)]
  \end{align*}
  Then we need to define six functors
  \[
    \begin{tikzpicture} \node(dummy) at (0,0.5) {};
      \node(X1) at (3,0) {$X_1$};
      \node(X2) at (0,0) {$X_2$};
      \node(X3) at (-3,0) {$X_3$};
      \draw[1cell, transform canvas={yshift=12pt}, swap]
      (X2) -- node {$r$} (X3);
      \draw[1cell, swap]
      (X2) -- node {$s$} (X3);
      \draw[1cell, transform canvas={yshift=-12pt}, swap]
      (X2) -- node {$t$} (X3);
      \draw[1cell, transform canvas={yshift=12pt}, swap]
      (X1) -- node(l) {$v$} (X2);
      \draw[1cell]
      (X2) -- node {$i$} (X1);
      \draw[1cell, transform canvas={yshift=-12pt}, swap]
      (X1) -- node {$w$} (X2);
    \end{tikzpicture}.
  \]
  Since this is very complicated and redundant,
  we omit the detail and give an outline.
  We define six functors as follows.
  \begin{align*}
    v_{AB}  & = T(A,-)  \colon \A(A,B)                \rightarrow [T(A,A), T(A,B)] \\
    w_{AB}  & = T(-,B)  \colon \A(A,B)                \rightarrow [T(B,B), T(A,B)] \\
    i_A     & = 1_A     \colon \one                   \rightarrow \A(A,A)          \\
    r_{ABC} & = T(A,-)  \colon \A(B,C)                \rightarrow [T(A,C), T(A,C)] \\
    s_{ABC} & = c_{ABC} \colon \A(B,C) \times \A(A,B) \rightarrow \A(A,C)          \\
    t_{ABC} & = T(-,C)  \colon \A(A,B)                \rightarrow [T(B,C), T(A,C)]
  \end{align*}
  Functors $v$, $w$, $i$, $r$, $s$, $t$ are those canonically constructed from these data.
  For example, $v$ is defined by the product of the functors
  \[
    X_1 \xrightarrow{\pi_A} T(A,A) \xrightarrow{\overline{v_{AB}}} [\A(A,B), T(A,B)]
  \]
  where $\overline{v_{AB}}$ is the transpose of $v_{AB}$.
  One can check that these six functors actually constitute coherence data,
  whose five 2-cells are all identities.

  Since $\Cat$ is complete, there is a descent object $(X,x,\xi)$ for this coherence data.
  Let $\lambda_A$ be the composite
  \[
    X\xrightarrow{x} X_1 \xrightarrow{\pi_A} T(A,A)
  \]
  and $\lambda_{AB}$ be the 2-cell
  \[
    \begin{tikzpicture}[x=15mm,y=10mm,baseline=(B.base)]
      \node(X) at (0,0) {$X$};
      \node(A1) at (1,1.5) {$X_1$};
      \node(B1) at (1,-1.5) {$X_1$};
      \node(AB1) at (2,0) {$X_2$};
      \node(A) at (3,1.5) {$T(A,A)$};
      \node(B) at (3,-1.5) {$T(B,B)$};
      \node(AB) at (4,0) {$[\A(A,B),T(A,B)]$};
      \draw[1cell] (X) to node {$x$} (A1);
      \draw[1cell,swap] (X) to node {$x$} (B1);
      \draw[1cell] (A1) to node {$v$} (AB1);
      \draw[1cell] (B1) to node {$w$} (AB1);
      \draw[1cell] (A1) to node {$\pi_A$} (A);
      \draw[1cell] (B1) to node {$\pi_B$} (B);
      \draw[1cell] (AB1) to node {$\pi_{AB}$} (AB);
      \draw[1cell] (A) to node {$v_{A}$} (AB);
      \draw[1cell,swap] (B) to node {$w_{B}$} (AB);
      \twocelld{0.95,0}{$\xi$}
    \end{tikzpicture}.
  \]
  One can show that, for each $A$ and $B$, this 2-cell $\lambda_{AB}$ corresponds to
  a family ${\{\lambda_f\}}_{f\colon A\rightarrow B}$
  \[
    \begin{tikzpicture}[x=12mm,y=8mm,baseline=(B.base)]
      \node(K) at (0,0) {$X$};
      \node(A) at (3,0) {$T(A,A)$};
      \node(B) at (0,-2) {$T(B,B)$};
      \node(AB) at (3,-2) {$T(A,B)$};
      \draw [1cell] (K) to node {$\lambda_A$} (A);
      \draw [1cell] (K) to node[swap] {$\lambda_B$} (B);
      \draw [1cell] (A) to node {$T(1_A,f)$} (AB);
      \draw [1cell] (B) to node[swap] {$T(f,1_B)$} (AB);
      \twocelldl[swap]{1.5,-1}{$\lambda_f$}
    \end{tikzpicture}
  \]
  which is natural in $f$ in the sense of \cref{eq:wedge_naturality}.
  This family $\lambda_f$ also satisfies each of \cref{eq:wedge_identity,eq:wedge_composition}
  because the equality of diagrams \cref{eq:lax_descent_small,eq:lax_descent_large}
  in the definition of descent objects corresponds respectively.
  Also, the universality of a lax end follows from that of limits.
\end{proof}

From this construction of lax ends with weighted limits, we deduce
the following two corollaries.

\begin{corollary}\label{cor:functoriality_laxend_cat}
  Let $T$ be a 2-functor $\A^\op\times\A\times\B\rightarrow\Cat$.
  There is a canonical 2-functor $\lend_A T(A,A,-)$, sending
  $B$ to $\lend_A T(A,A,B)$.
\end{corollary}
\begin{proof}
  It follows from \Cref{prop:existence_in_cat} and the functoriality of weighted limits.
\end{proof}

\begin{corollary}\label{cor:end_laxend_commute}
  Lax ends commutes with weighted limits
  \begin{align}
    \wlim{F}{\left(\lend_A T(A,A,-)\right)} \iso \lend_A {\wlim{F}{T(A,A,-)}}.\label{eq:end_limit_commute}
  \end{align}
  As a special case, for $T \colon \A^\op \times \B^\op \times \A \times \B \rightarrow \Cat$,
  \begin{align}
    \int_B\lend_A T(A,B,A,B) \iso \lend_A\int_B T(A,B,A,B).\label{eq:end_laxend_commute}
  \end{align}
\end{corollary}
\begin{proof}
  It follows from \Cref{prop:existence_in_cat} and the commutativity of weighted limits.
\end{proof}

Note that, for coends and colimits, \Cref{cor:functoriality_laxend_cat,cor:end_laxend_commute} have duals.
Also, Fubini's rule for lax ends is proved as follows.

\begin{proposition}\label{prop:Fubini}
  Let $T\colon \A^\op\times\B^\op\times\A\times\B\rightarrow\Cat$.
  Fubini's theorem holds for both lax ends and coends.
  \begin{align}
    \lend_{A,B} T(A,B,A,B) \iso \lend_A\lend_B T           & (A,B,A,B) \iso \lend_B\lend_A T(A,B,A,B)
    \label{eq:Fubini_lax_end}                                                                                  \\
    \lend^{A,B}\!T(A,B,A,B) \iso \lend^A\!\!\!\!\lend^B\!T & (A,B,A,B) \iso \lend^B\!\!\!\!\lend^A\!T(A,B,A,B)
    \label{eq:Fubini_lax_coend}
  \end{align}
\end{proposition}
\begin{proof}
  It suffices to prove the left isomorphism of \cref{eq:Fubini_lax_end}.
  The right isomorphism follows from the left, and \cref{eq:Fubini_lax_coend} is the dual.

  From the universality of $\lend_A$, a functor $X\rightarrow\lend_A\lend_B T(A,B,A,B)$ corresponds to a lax wedge
  $X\rightarrow \lend_B T(-,\!B,\!-,\!B)$, which is a pair
  $\{\tau_{A}\}$ and ${\{\tau_f\}}_{f\colon A\rightarrow A'}$ as
  \[\begin{tikzpicture}[x=17mm,y=7mm,baseline=(AB.base)]
      \node(K) at (0,0) {$X$};
      \node(A) at (3,0)   {$\lend_B T(A,\!B,\!A,\!B)$};
      \node(B) at (0,-2)  {$\lend_B T(A'\!,\!B,\!A'\!,\!B)$};
      \node(AB) at (3,-2) {$\lend_B T(A,\!B,\!A'\!,\!B)$};
      \draw [1cell] (K) to node {$\tau_A$} (A);
      \draw [1cell] (K) to node[swap] {$\tau_{A'}$} (B);
      \draw [1cell] (A) to node {$\lend_B T(1,\!1,\!f,\!1)$} (AB);
      \draw [1cell] (B) to node[swap] {$\lend_B T(f,\!1,\!1,\!1)$} (AB);
      \twocelldl{1.5,-1}{$\tau_f$}
    \end{tikzpicture}.
  \]
  Again, from the universality of $\lend_B$, each $\tau_A$ corresponds to a wedge from $X$ to $T(A,\!-,\!A,\!-)$.
  We now have a family of 1-cells $\{\sigma_{AB}\colon X \rightarrow T(A,\!B,\!A,\!B)\}$,
  families of 2-cells $\{\sigma_{Ag}\}$ making $\sigma_{A-}$ a wedge,
  and the family $\{\tau_f\}$.

  By defining $\sigma_{fB}$ as below,
  \[
    \begin{tikzpicture}[x=18mm,y=10mm,baseline=(BC.base)]
      \node(K) at (0,0) {$X$};
      \node(A) at (1,1.73) {$\lend_B T(A,\!B,\!A,\!B)$};
      \node(C) at (1,-1.73) {$\lend_B T(A'\!,\!B,\!A'\!,\!B)$};
      \node(AB) at (3,1.73) {$T(A,\!B,\!A,\!B)$};
      \node(BC) at (3,-1.73) {$T(A'\!,\!B,\!A'\!,\!B)$};
      \node(AC) at (4,0) {$ T(A,\!B,\!A'\!,\!B)$};
      \node(B) at (2,0) {$\lend_B T(A,\!B,\!A'\!,\!B)$};
      \draw [1cell] (K) to node {$\sigma_A$} (A);
      \draw [1cell] (K) to node[swap] {$\sigma_C$} (C);
      \draw [1cell] (A) to node        {} (AB);
      \draw [1cell] (C) to node[swap]  {} (BC);
      \draw [1cell] (AB) to node {$T(1,\!1,\!f,\!1)$} (AC);
      \draw [1cell] (BC) to node[swap] {$T(f,\!1,\!1,\!1)$} (AC);
      \draw [1cell] (A) to node {$\lend_B T(1,\!1,\!f,\!1)$} (B);
      \draw [1cell] (C) to node[swap] {$\lend_B T(f,\!1,\!1,\!1)$} (B);
      \draw [1cell] (B) to node[]     {} (AC);
      \twocelld[swap]{1,0}{$\tau_f$}
    \end{tikzpicture}
  \]
  $\sigma_{-B}$ also becomes a wedge.
  From the universality of $\lend_B T(A,B,A'\!,B)$, $\sigma_{fB}$ defines a modification of wedges.
  Writing down the of requirements for $\sigma_{fB}$ to be a modification,
  it turns out that $\sigma_{Ag}$ and $\sigma_{fB}$ are compatible in the sense of the condition in
  \Cref{prop:lax_naturality_for_each_variable}.
  Thus, we deduce that $\sigma$ is a wedge $X \rightarrow T$.
  Since we have proven that the universality for $\lend_A\lend_B T(A,\!B,\!A,\!B)$
  coincides with that for $\lend_{A,B} T(A,\!B,\!A,\!B)$, these are isomorphic.
\end{proof}

\label{section:lax_end_3}

\section{Lax limits, lax end calculus}\label{section:lax_limits_4}

In the previous section, we established some basic isomorphisms for lax ends.
Now, we will see more advanced results for lax ends, including the relation with lax limits.

The first important theorem is \Cref{thm:adjunction_functor_category_cat},
which establishes two adjunctions between a presheaf 2-category $[\A^\op,\Cat]$
and lax presheaf 2-category $\Lax{\A^\op,\Cat}$, which are lax morphism classifier and coclassifier.
\[
  \begin{tikzpicture}
    \node(p) at (0,0) {$[\A^\op,\Cat]$};
    \node(l) at (4,0) {$\Lax{\A^\op,\Cat}$};
    \draw[right hook->] (p) to node {} (l);
    \draw[1cell, bend right, swap] (l) to node {${(-)}^\sharp$} (p);
    \draw[1cell, bend left] (l) to node {${(-)}^\flat$} (p);
    \node at (2,0.4) {$\bot$};
    \node at (2,-0.4) {$\bot$};
  \end{tikzpicture}
\]

To this end, we first define this ${(-)}^\sharp$ and ${(-)}^\flat$.

\begin{definition}
  Let $F \colon \A^\op \rightarrow \Cat$ be a 2-functor. We define 2-functors $F^\sharp$ and $F^\flat$ from $\A^\op$ to $\Cat$ as follows:
  \begin{align}
    F^\sharp & \colon \lend^A \A(-,A) \times FA \label{eq:def_sharp} \\
    F^\flat  & \colon \lend_A [\A(A,-), FA] \label{eq:def_flat}
  \end{align}
\end{definition}

Now that we have already proven some useful isomorphisms, it just suffices to combine them.

\begin{theorem}\label{thm:adjunction_functor_category_cat}
  The inclusion $[\A^\op, \Cat] \rightarrow \Lax{\A^\op, \Cat}$ has both a left adjoint ${(-)}^\sharp$ and a right adjoint ${(-)}^\flat$.
  \begin{align}
    \Lax{\A^\op,\Cat} (F,H) & \iso
    [\A^\op, \Cat](F^\sharp,H)     \label{eq:adjunction_cat_sharp} \\
    \Lax{\A^\op,\Cat} (H,F) & \iso
    [\A^\op, \Cat](H, F^\flat)     \label{eq:adjunction_cat_flat}
  \end{align}
\end{theorem}
\begin{proof}
  For the right adjoint,
  \begin{align*}
    [\A^\op, \Cat](H, F^\flat)
     & \iso \int_C [HC, F^\flat C]                               \\
     & \iso \int_C \left[HC, \lend_A[\A(A,C), FA]\right]
     & \mathrm{by\,\cref{eq:def_flat}}                           \\
     & \iso \int_C\lend_A[HC, [\A(A,C), FA]]
     & \mathrm{by\,\cref{eq:cat_lax_end_commute_representable}}  \\
     & \iso \lend_A\int_C[\A(A,C) \times HC, FA]
     & \mathrm{by\,\cref{eq:end_laxend_commute}}                 \\
     & \iso \lend_A\left[\int^C \A(A,C)\times HC, FA\right]      \\
     & \iso \lend_A[HA, FA]
     & \mathrm{by\,Yoneda\,Lemma}                                \\
     & \iso \Lax{\A^\op,\Cat} (H,F)
     & \mathrm{by\,\cref{eq:hom_of_functor_category_as_lax_end}}
  \end{align*}
  And the left adjoint is the dual.
\end{proof}

The left adjoint ${(-)}^\sharp$ we established above is known to give a representation of lax limits $\laxlim^F G$
with usual weighted limits $\wlim{F^\sharp}{G}$.
We start by recalling the definition of lax limits and then show this statement.

\begin{definition}
  Let $F \colon \A \rightarrow \Cat$ and $G\colon \A \rightarrow \K$ be 2-functors.
  A \emph{lax limit} is a representing object $\laxlim^F G$
  of $\Lax{\A, \Cat}(F, \widehat G)$.
  \begin{equation}
    \K(K, \laxlim^F G) \iso \Lax{\A, \Cat}(F, \K(K, G-))\label{eq:def_lax_limit}
  \end{equation}
  Dually, for $F\colon\A^\op\rightarrow\Cat$ and $G\colon\A\rightarrow\K$,
  a \emph{lax colimit} is a representing object $\laxcolim^F G$
  of $\Lax{\A^\op, \Cat}(F,\tilde G)$.
  \begin{equation}
    \K(\laxcolim^F G, K) \iso \Lax{\A^\op, \Cat}(F, \K(G-, K))\label{eq:def_lax_colimit}
  \end{equation}
\end{definition}

If $\K = \Cat$, we can calculate the right-hand side of \cref{eq:def_lax_limit} as
\begin{align*}
  \Lax{A, Cat}(F, [X, G-])
   & \iso \lend_{A} [FA, [X, GA]]
   & \mathrm{by\,\cref{eq:hom_of_functor_category_as_lax_end}} \\
   & \iso \lend_{A} [X, [FA, GA]]                              \\
   & \iso \left[X, \lend_{A} [FA, GA]\right]
   & \mathrm{by\,\cref{eq:cat_lax_end_commute_representable}}  \\
   & \iso [X, \Lax{\A, \Cat}(F, G)].
   & \mathrm{by\,\cref{eq:hom_of_functor_category_as_lax_end}}
\end{align*}
Therefore, by Yoneda lemma, we have
\begin{align}
  \laxlim^F G \iso
  \lend_{A} [FA, GA] \iso
  \Lax{\A, \Cat}(F, G) \label{eq:laxlim_in_cat}
\end{align}
for any $F, G \in [\A, \Cat]$.
And, with \Cref{thm:adjunction_functor_category_cat},
\begin{align}
  \laxlim^F G
  \iso [\A, \Cat](F^\sharp, G)
  \iso \wlim{F^\sharp}{G}, \label{eq:cat_laxlim_with_lim}
\end{align}
or conversely,
\begin{align}
  \laxlim^F G
  \iso [\A, \Cat](F, G^\flat)
  \iso \wlim{F}{G^\flat}. \label{eq:cat_laxlim_with_lim_flat}
\end{align}
Isomorphisms \cref{eq:cat_laxlim_with_lim,eq:cat_laxlim_with_lim_flat} show that every lax limits in $\Cat$ can be represented by weighted limits with the same weights or diagrams.
We show \cref{eq:cat_laxlim_with_lim} can be generalized in general 2-categories.

\begin{proposition}\label{prop:laxlim_with_lim}
  Let $G \colon \A \rightarrow \K$ be a 2-functor. Then
  \begin{align}
    \laxlim^F G   & \iso \wlim{F^\sharp}{G},\label{eq:laxlim_with_limit}       \\
    \laxcolim^F G & \iso \wcolim{F^\sharp}{G}\label{eq:laxcolim_with_colimit},
  \end{align}
  whenever they exist.
\end{proposition}
\begin{proof} \Cref{eq:laxlim_with_limit} follows from
  \begin{align*}
    \K(K, \laxlim^F G)
     & \iso \Lax{\A,\Cat}(F, \widehat GK)    & \mathrm{by\,\cref{eq:def_lax_limit}}        \\
     & \iso [\A,\Cat](F^\sharp, \widehat GK) & \mathrm{by\,\cref{eq:adjunction_cat_sharp}} \\
     & \iso \K(K, \wlim{F^\sharp}{G}).
  \end{align*}
  \Cref{eq:laxcolim_with_colimit} is similar.
\end{proof}

Note that, by similar arguments as in \cref{eq:laxlim_in_cat,prop:laxlim_with_lim}, the lax colimits in $\Cat$ can be presented
in the following several other ways:
\begin{align*}
  \laxcolim^F G & \iso \lend^A FA \times GA \\
                & \iso \lend^A GA \times FA \\
                & \iso \laxcolim^G F        \\
                & \iso \wcolim{G^\sharp}{F}
\end{align*}

We can directly deduce the commutativity of lax ends from the commutativity of weighted limits, which is the right isomorphism in Fubini' theorem \ref{prop:Fubini}.
And from Fubini's theorem and the representation of lax limits in $\Cat$ with lax ends \cref{eq:laxlim_in_cat},
we deduce that lax limits commute.

\begin{corollary}
  Let $F\colon \A\rightarrow\Cat$, $F'\colon\B\rightarrow\Cat$, and $G\colon\A\times\B\rightarrow\Cat$.
  Then,
  \begin{align*}
    \laxlim^F \laxlim^{F'} G     & \iso \laxlim^{F'} \laxlim^F G      \\
    \laxcolim^F \laxcolim^{F'} G & \iso \laxcolim^{F'} \laxcolim^F G.
  \end{align*}
\end{corollary}

\begin{example}
  One of the simple lax limits is the conical lax limit, that is,
  a lax limit weighted by the constant functor $\Delta\one\colon\A\rightarrow\Cat$.
  By \Cref{prop:laxlim_with_lim}, this is $\laxlim^{\Delta\one} G \iso \wlim{{(\Delta\one)}^\sharp}{G}$.
  Let us examine the weight $(\Delta\one)^\sharp$ of the right-hand side.
  By the definition of ${(-)}^\sharp$, $(\Delta\one)^\sharp$ is a 2-functor $\lend^A \A(A,-)$, sending $C\in\A$ to
  $\lend^A \A(A,C) \iso \laxcolim^{\Delta\one}\A(-,C)$,
  In fact, $(\Delta\one)^\sharp C$ turns out to be the \emph{lax slice category} $\A/C$,
  whose object is a 1-cell into $C$, and whose morphism from $p$ to $q$ is a pair $(f,\bar f)$ as follows.
  \[\begin{tikzpicture}[x=10mm,y=7mm]
      \node(A) at (-1,0) {$A$};
      \node(A') at (1,0) {$A'$};
      \node(C) at (0,-2) {$C$};
      \draw[1cell] (A) to node {$f$} (A');
      \draw[1cell,swap] (A) to node {$p$} (C);
      \draw[1cell] (A') to node {$q$} (C);
      \twocellr{0,-1}{$\bar f$}
    \end{tikzpicture}\]
  To show this actually is a lax colimit, we then construct the unit lax cocone $\{\lambda_A\colon\A(A,C) \rightarrow \A/C\}$
  with ${\{\lambda_f\colon \lambda_A\circ\A(f,C)\rightarrow\lambda_{A'}\}}_{f\colon A\rightarrow A'}$ by
  \begin{align*}
     & \lambda_A(p\colon A\rightarrow C) = p,
    \quad
    \lambda_A(\alpha\colon p\rightarrow p') = (1_A, \alpha),                    \\
     & {(\lambda_f)}_{q\colon A'\rightarrow C} = (f,1) \colon qf \rightarrow q.
  \end{align*}
  We need to show its universality. To this end, let us assume there is another lax cocone
  $\{\sigma_A\colon \A(A,C)\rightarrow X\}$ with $\{\sigma_f\colon\sigma_A\circ\A(f,C)\rightarrow\sigma_{A'}\}$,
  and a morphism of cocone $u\colon \A/C \rightarrow X$. Then, these data must satisfy the following:
  \begin{align*}
    \begin{tikzpicture}
      [x=10mm,y=7mm,baseline=(l.base)]
      \node(A) at (0,0) {$\A(A,C)$};
      \node(S) at (0,-2) {$\A/C$};
      \node(X) at (2,-2) {$X$};
      \draw[1cell,swap] (A) to node(l){$\lambda_A$} (S);
      \draw[1cell] (A) to node {$\sigma_A$} (X);
      \draw[1cell] (S) to node {$u$} (X);
    \end{tikzpicture}\quad & \quad\text{commutes, and} \\
    \begin{tikzpicture}
      [x=10mm,y=9mm,baseline=(l.base)]
      \node(A) at (-2,-1.8) {$\A(A,C)$};
      \node(A') at (0,0) {$\A(A',C)$};
      \node(S) at (0,-1.8) {$\A/C$};
      \node(X) at (2,-1.8) {$X$};
      \draw[1cell] (A') to node(l){$\lambda_{A'}$} (S);
      \draw[1cell,swap] (A) to node(l){$\lambda_{A'}$} (S);
      \draw[1cell] (A') to node(l) {$\sigma_{A'}$} (X);
      \draw[1cell] (S) to node {$u$} (X);
      \draw[1cell,swap] (A') to node {$\A(f,C)$} (A);
      \twocelld{-0.7,-1.2}{$\lambda_f$}
    \end{tikzpicture}
    \quad                           & =\quad
    \begin{tikzpicture}
      [x=10mm,y=9mm,baseline=(l.base)]
      \node(A) at (-2,-1.8) {$\A(A,C)$};
      \node(A') at (0,0) {$\A(A',C)$};
      \node(S) at (0,-1.8) {$\A/C$};
      \node(X) at (2,-1.8) {$X$.};
      \draw[1cell] (A') to node(l) {$\sigma_{A'}$} (X);
      \draw[1cell] (S) to node {$u$} (X);
      \draw[1cell,swap] (A') to node {$\A(f,C)$} (A);
      \draw[1cell,swap] (A) to node{$\lambda_{A'}$} (S);
      \twocelld{0,-0.9}{$\sigma_f$}
    \end{tikzpicture}
  \end{align*}
  Therefore, we need to have
  \begin{itemize}
    \item for each $p\colon A\rightarrow C$, $u(p) = \lambda_A(p) = \sigma_A(p)$,
    \item for each $\alpha \colon p \rightarrow q$ in $\A(A,C)$,
          $u((1,\alpha)) = \lambda_A(\alpha) = \sigma_A(\alpha)$,
    \item for each $f\colon A \rightarrow A'$ and $q\colon A'\rightarrow C$,
          $u((f,1)) = u{(\lambda_f)}_q = {(\sigma_f)}_q$.
  \end{itemize}
  Since an arrow $(f,\bar f)\colon p\rightarrow q$ in $\A/C$ is the composite $p\xrightarrow{(1,\bar f)}qf\xrightarrow{(f,1)}q$,
  the data of $u$ is fully determined by $\sigma$.
  This proves the uniqueness of $u$, and the existence of $u$ is checked by some diagram chasing.

  When $\A$ is a category $\two$ for example, the resulting weight ${(\Delta \one)}^\sharp$ sends
  the diagram $0\rightarrow 1$ to $\one\xrightarrow{0}\two$.
\end{example}

In enriched category theory, $\int_A T(A,A)$ could be rewritten as $\wlim{\Hom_\A}{T}$.
Correspondingly, we wish to rewrite lax ends using limits or lax limits.
The reader might naturally think of $\laxlim^{\Hom_\A} T$ as a candidate for the rewriting.
However, it turns out to be incorrect.
In fact, calculating $[X,\laxlim^{\Hom_\A}T]$ yields $\lend_A\lend_C \left[\A(C,A), \left[X, T(C,A)\right]\right]$,
which is different from $[X,\lend_AT(A,A)] \iso \lend_A\int_C\left[\A(C,A), \left[X,T(C,A)\right]\right]$.

The next theorem gives the correct answer.

\begin{theorem}\label{thm:representation_of_laxend_by_limit}
  Let $T \colon \A^\op \times \A \rightarrow \Cat$ be a 2-functor. Then
  \begin{align}
    \lend_A T(A,A) & \iso \wlim{{(Y-)}^\sharp ?}{T(?,-)}   \label{eq:lend_with_wlim}     \\
    \lend^A T(A,A) & \iso \wcolim{{(Y-)}^\sharp ?}{T(?,-)} \label{eq:lcoend_with_wcolim}
  \end{align}
  where ${(Y-)}^\sharp ?$ is a profunctor sending $(A,A') \in \A^\op \times \A$ to ${(YA')}^\sharp A$.
\end{theorem}
\begin{proof}
  \cref{eq:lend_with_wlim} follows from the following isomorphisms. \cref{eq:lcoend_with_wcolim} is similar.
  \begin{align*}
     & \left[X, \wlim{{(Y-)}^\sharp ?}{T(?,-)}\right]                                      \\
     & \iso \int_{C,C'}[{(YC')}^\sharp C, [X, T(C,C')]]                                    \\
     & = \int_{C,C'}\left[\lend^A \A(C,A) \times \A(A,C'), [X, T(C,C')]\right]
     & \mathrm{by\,\cref{eq:def_sharp}}                                                    \\
     & \iso \int_{C,C'}\lend_A\left[\A(C,A) \times \A(A,C'), [X, T(C,C')]\right]
     & \mathrm{by\,the\,dual\,of\,\cref{eq:cat_lax_end_commute_representable}}             \\
     & \iso \lend_A\int_C\int_{C'}\left[\A(C,A), \left[\A(A,C'), [X,T(C,C')]\right]\right]
     & \mathrm{by\,\cref{eq:end_laxend_commute}}                                           \\
     & \iso \lend_A\int_C\left[\A(C,A), \int_{C'}\left[\A(A,C'), [X,T(C,C')]\right]\right] \\
     & \iso \lend_A\int_C\left[\A(C,A), \left[X,T(C,A)\right]\right]
     & \mathrm{by\,Yoneda\,lemma}                                                          \\
     & \iso \lend_A\left[X,T(A,A)\right]
     & \mathrm{by\,Yoneda\,lemma}                                                          \\
     & \iso \left[X, \lend_AT(A,A)\right].
     & \mathrm{by\,\cref{eq:cat_lax_end_commute_representable}}
  \end{align*}
\end{proof}

\section{Lax ends in general 2-Categories}\label{section:general_lax_end_5}

In the previous two sections, we examined lax ends in $\Cat$, but of course, lax ends can be defined in general 2-categories,
which we discuss in this section.
The problem is that, while there are several characterizations of the lax ends in $\Cat$, it is not obvious which one should be used for the generalization.
Here, we adopt its commutativity with representables as the definition for the general case and show that several other characterizations coincide.

\begin{definition}
  Let $T$ be a 2-functor $T \colon \A^\op \times \A \rightarrow \K$.
  We define \emph{lax [co]ends in $\K$} as representing objects
  \begin{align}
    \K\left(B, \lend_A T(A,A)\right) & \iso \lend_A \K(B, T(A,A)), \label{eq:def_general_laxend} \\
    \K\left(\lend^A T(A,A), B\right) & \iso \lend_A \K(T(A,A), B). \nonumber
  \end{align}
\end{definition}

Firstly, we check the universality of unit in \cref{eq:def_general_laxend}.

\begin{proposition}
  Let $T$ be a 2-functor $T \colon \A^\op \times \A \rightarrow \K$.
  Then, $T$ has a lax end in $\K$ if and only if there is a lax wedge
  that has the same universality as lax ends in $\Cat$.
\end{proposition}
\begin{proof}
  As we observed in \Cref{prop:cat_lax_end_commute_representable},
  there is also a bijection between
  the set of lax wedges $K \rightarrow T$ in $\K$ and
  the set of lax wedges $\one\rightarrow\K(B,T(?-))$ in $\Cat$.
  Moreover, those modifications precisely coincide.
\end{proof}

Let $T\colon\A^\op\times\A\times\B\rightarrow\K$ be a 2-functor.
To verify the notation $\lend_A T(A,A,-)$ defined in the case $\K=\Cat$ is compatible with the definition of the general lax end,
we need to check that the lax end in a functor category is computed pointwise.
This is checked by the following isomorphism for an arbitrary $F\colon \B\rightarrow \K$.
\begin{align*}
  \lend_A \left[B,\K\right] (F, T(A,A,-))
   & \iso \lend_A \int_B \K\left(FB,T(A,A,B)\right)           \\
   & \iso \int_B \lend_A \K\left(FB,T(A,A,B)\right)
   & \mathrm{by~\cref{eq:end_laxend_commute}}                 \\
   & \iso \left[\B,\K\right]\left(F-,\lend_A T(A,A,-)\right).
\end{align*}
Therefore, we can use the notation $\lend_AT(A,A,-)$ freely.

The other characterization of the lax limits in $\Cat$ was a weighted limit $\wlim{{(Y-)}^\sharp ?}{T(?,-)}$, proved in \Cref{thm:representation_of_laxend_by_limit}.
Its generalization to arbitrary 2-categories again can be shown to agree with the definition above by similar calculation: for general $T\colon\A^\op\times\A\rightarrow\K$,
\begin{align*}
  \lend_A T(A,A) & \iso \wlim{{(Y-)}^\sharp ?}{T(?,-)}    \\
  \lend^A T(A,A) & \iso \wcolim{{(Y-)}^\sharp ?}{T(?,-)}.
\end{align*}
As a corollary, it is proved that, if $\K$ is [co]complete, then $\K$ admits lax [co]ends.

By the definition of lax ends in $\K$ and \Cref{cor:end_laxend_commute},
lax [co]ends in general $\K$ also commute with weighted [co]limits.
\begin{align}
  \wlim{F}{\left(\lend_AT(A,A,-)\right)} \iso \lend_A\wlim{F}{T(A,A,-)}\label{eq:general_laxend_commute_limit}
\end{align}
Also, the Fubini's rule follows from the $\Cat$ case (\Cref{prop:Fubini}).

Next, we would like to show the adjunctions between
functor categories with 2/lax transformations in general $\K$ again.
As in the $\Cat$ case, we define ${(-)}^\sharp$ and ${(-)}^\flat$.

\begin{definition}
  Let $F\colon\A\rightarrow\K$ be a 2-functor.
  We define $F^\sharp, F^\flat\colon\A\rightarrow\K$ as follows.
  \begin{align}
    F^\sharp & \colon \lend^A \A(A,-) \copower FA \label{eq:def_sharp_general} \\
    F^\flat  & \colon \lend_A \A(-,A) \power FA \label{eq:def_flat_general}
  \end{align}
\end{definition}

Note that a general $\K$ do not admit ${(-)}^\flat$[or ${(-)}^\sharp$],
since $\K$ might not have [co]limits.

And by the same calculation as in the $\Cat$ case,
the next theorem holds.

\begin{theorem}
  If $\K$ is complete, then the inclusion $[\A,\K]\rightarrow\Lax{\A,\K}$ has
  a right adjoint ${(-)}^\flat$.
  Dually, if $\K$ is cocomplete, then the inclusion $[\A,\K]\rightarrow\Lax{\A,\K}$ has
  a left adjoint ${(-)}^\sharp$.
\end{theorem}

We then would like to return to the topic of lax limits.
Let $\K$ be a complete 2-category, and
$F\colon\A\rightarrow \K$, $G\colon\A\rightarrow\K$ be 2-functors.
By similar isomorphisms just before \cref{eq:laxlim_in_cat}, we have an isomorphism,
\begin{align*}
  \Lax{\A,\Cat}(F,\K(K,G-))
   & \iso \lend_A [FA,\K(K,GA)]
   & \mathrm{by~\cref{eq:hom_of_functor_category_as_lax_end}} \\
   & \iso \lend_A \K\left(K,FA\power GA\right)                 \\
   & \iso \K\left(K,\lend_AFA\power GA\right).
   & \mathrm{by~\cref{eq:def_general_laxend}}
\end{align*}
Therefore, a lax limit can also be presented as,
\begin{align}
  \laxlim^F G \iso \lend_A FA\power GA.\label{eq:laxlim_with_laxend_and_cotensor}
\end{align}
\Cref{prop:laxlim_with_lim} showed that this could be presented as the limit $\wlim{F^\sharp}{G}$.
However, since we did not define $G^\flat$ for general $\K$ in the previous section,
the isomorphism
\begin{align}\label{eq:general_laxlim_with_lim_flat}
  \laxlim^F G \iso \wlim{F}{G^\flat}
\end{align}
in \cref{eq:cat_laxlim_with_lim_flat} did not make sense at that time.
But now, since we defined $G^\flat$ in \cref{eq:def_flat_general}, one can check
\cref{eq:general_laxlim_with_lim_flat} by
\begin{align*}
  \K\left(K, \wlim{F}{G^\flat}\right)
   & \iso \left[\A, \Cat\right] \left(F, \K\left(K,G^\flat-\right)\right) \\
   & \iso \left[\A, \Cat\right] \left(F, \K\left(K,G-\right)^\flat\right) \\
   & \iso \Lax{\A, \Cat} (F, \K(K, G-)),
\end{align*}
where the second isomorphism is from the fact that representables preserves lax ends and powers.

\section{The class of limits where lax ends live}\label{section:pie_lend_llim_6}

The classification of 2-categorical limits is an interesting and complicated problem.
To classify them, several classes of limits have been invented, such as the flexible limits \cite{bird1989flexible} and the PIE limits \cite{power_robinson_1991_pie}.
In this section, we show that a lax end is a PIE limit, but not a lax limit.
We first show the former part: a lax end is a PIE limit.

\begin{theorem}
  If a 2-category $\K$ admits products, inserters, and equifiers,
  Then it admits all lax ends and lax limits.
\end{theorem}
\begin{proof}
  In the same way as we did in \Cref{prop:existence_in_cat},
  we can define $X_1$, $X_2$, $X_3$ as follows,
  \begin{align*}
    X_1 & = \prod_{A\in\A} T(A,A)                                \\
    X_2 & = \prod_{A,B\in\A} \A(A,B) \power T(A,B)                \\
    X_3 & = \prod_{A,B,C\in\A} \A(B,C)\times\A(A,B) \power T(A,C)
  \end{align*}
  with six 1-cells in coherence data in the same manner,
  and check the five identities.
  Since $\K$ has inserters and equifiers,
  a descent object of this coherence data does exist in $\K$,
  which is the desired lax end in $\K$.

  The isomorphism \cref{eq:laxlim_with_laxend_and_cotensor} shows that,
  since $\K$ admits lax ends and powers, it also admits lax limits.
\end{proof}

Then, we show that there is a 2-category with all lax limits but not all lax ends.

\begin{lemma}
  A 2-category $\K$ with lax ends admits powers, all lax limits and oplax limits.
\end{lemma}
\begin{proof}
  Let $K\in \K$ and $\mathbb{A}$ be a category.
  The lax end of the constant functor $\Delta K\colon \mathbb{A}^\op\times\mathbb{A}\rightarrow \K$ determines the power $\mathbb{A}\power K$ since we have an isomorphism
  \begin{align*}
    \K\left(X, \lend_{A\in\mathbb{A}} \Delta K\right) \iso
    \lend_{A\in\mathbb{A}} \K(X,K) \iso
    [\mathbb{A}, \K(X,K)]
    \iso \K(X, \mathbb{A}\power K).
  \end{align*}
  Therefore, from \cref{eq:laxlim_with_laxend_and_cotensor}, $\K$ has all lax limits.

  As we mentioned in \Cref{section:lax_naturality_2}, an oplax wedge and end of $T$ can be
  defined as a lax wedge and end of another $T'$ with $T'(A,B) = T(A,B)$.
  And therefore,
  writing oplax ends with $\oplend$,
  $\K$ also admits oplax limits $\oplend_A FA\power GA\iso\oplaxlim^F G$.

\end{proof}

\begin{theorem}
  There is a 2-category with all lax limits, but does not have lax ends.
\end{theorem}
\begin{proof}
  Let $\K$ be a full sub 2-category of $\Cat$ whose objects are categories with finite products.
  Since there is a 2-monad $T$ on $\Cat$ which has $\K$ as the 2-category of algebras and oplax morphisms,
  $\K$ has all lax limits, which was shown in \cite{Lack2005,Lack_Shulman}.

  However, this $\K$ lucks oplax limits of an arrow.
  To see this, first observe that $\K$ is dense in $\Cat$,
  which is because $\One$ is dense in $\Cat$ as a $\Cat$-enriched category.
  And therefore all the limits in $\K$ are those in $\Cat$.
  Let $\mathbb{B}$ a category with finite products and $b\in \mathbb{B}$.
  Then, the oplax limit of the arrow $\one\xrightarrow{b}\mathbb{B}$ in $\Cat$
  is the slice category $\mathbb{B}/b$.
  Since the product in slice category is the pullback,
  the oplax limit is not included in $\K$ in general.
\end{proof}

\section{Pseudo ends and bicategorical (co)Yoneda lemma}\label{section:yoneda_7}

As is widely known, the theory of bicategories has its version of the Yoneda lemma, the bicategorical Yoneda lemma.
For simplicity, we restrict to 2-categories and 2-functors here.
\begin{theorem}[bicategorical Yoneda lemma]
  Let $F\colon \A^\op \rightarrow\Cat$ be a 2-functor. Then there is a following equivalence of categories,
  which exists the pseudo-natural in $A \in \A$ and $F \in \Ps{\A^\op, \Cat}$.
  \begin{align}
    \Ps{\A^\op, \Cat}(YA, F) \eqv FA
    \label{lem:bicategorical_yoneda}
  \end{align}
\end{theorem}
For the proof, consult other literature such as \cite{johnson2021}.

Clearly, since everything proved in \Cref{section:lax_naturality_2,section:lax_end_3,section:lax_limits_4} has its counterpart in pseudo case
with pseudo transformations/ends/limits, etc.
We denote pseudo ends by integral with tilde $\psend_A$, and adjunctions corresponding to \Cref{thm:adjunction_functor_category_cat} by
double sharp $\dsharp$ and double flat $\dflat$.
\begin{align}
  \Ps{\A^\op,\Cat} (F,H) & \iso
  [\A^\op, \Cat](F^\dsharp,H)     \label{eq:adjunction_cat_dsharp} \\\nonumber
  \Ps{\A^\op,\Cat} (H,F) & \iso
  [\A^\op, \Cat](H, F^\dflat)     \label{eq:adjunction_cat_dflat}
\end{align}
Since the hom-category of $\Ps{\A^\op,\Cat}$ can be represented with pseudo ends, the bicategorical Yoneda lemma is equivalent to saying
\begin{align*}
  \psend_{C\in\A} [\A(C, A), FC] \eqv FA.
\end{align*}
Since the left-hand side is the definition of $F^\dflat A$, there is an equivalence $F^\dflat A \eqv FA$.

This equivalence concludes the equivalence of functors $F\eqv F^\dflat$.
However, it should be noted that this is one that in $\Ps{\A^\op, \Cat}$ and not in $[\A^\op, \Cat]$,
since the bicategorical Yoneda lemma is only pseudo-natural in $A$.

In the remaining part of this section,
we show the bicategorical coYoneda lemma $F^\dsharp A \eqv FA$
as the dual for the bicategorical Yoneda lemma $F^\dflat A\eqv FA$,
which is also shown in \cite{Alexander2016coyoneda}.

\begin{theorem}[bicategorical coYoneda lemma]
  There exists the following equivalence of categories which is pseudo natural for $A\in\A$ and 2-natural for $F\in\Ps{\A^\op, \Cat}$,
  \[
    F^\dsharp A = \psend^{C\in\A} \A(A, C) \times FC \eqv FA.
  \]
\end{theorem}
\begin{proof}
  Let the following be the universal cowedges for the end or the pseudo end.
  \begin{align*}
    \lambda_{CA}            \colon  \A(C,A) \times FC \rightarrow \int^C   \A(C,A) \times FC \\
    \widetilde{\lambda}_{CA}\colon  \A(C,A) \times FC \rightarrow \psend^C \A(C,A) \times FC
  \end{align*}
  Note that these are both 2-natural for $A$.

  Since 2-wedge $\lambda_{-A}$ is also a pseudo wedge, from the universality of pseudo ends,
  there is a unique functor $\varepsilon_A\colon F^\dsharp A\rightarrow FA$
  satisfying $\varepsilon_A \wtilde{\lambda}_{CA} = \lambda_{CA}$.
  \[
    \begin{tikzcd}
      {\A(A,C) \times FC}
      \arrow[d,  "{\widetilde{\lambda}_{CA}}"]
      \arrow[dr, "{{\lambda}_{CA}}"]
      \arrow[drr,"{\overline{F_{AC}}}", bend left=10] \\
      {F^\dsharp A = \psend^{C\in\A} \A(A, C)\times FC} \rar["{\varepsilon_A}"'] &
      {\int^{C\in\A} \A(A, C)\times FC} \rar[phantom, "\iso"]
      &
      {FA}
    \end{tikzcd}
  \]
  This $\varepsilon_A$ is the counit for the left adjoint ${(-)}^\dsharp$.
  Since $\widetilde{\lambda}_{AD}$ is 2-natural for $D$,
  the transpose $\alpha_{AD} \colon FA \rightarrow [\A(A,D), F^\dsharp D]$
  is also 2-natural for $D$. So, $\alpha_{AD}$ induces
  $\eta_A\colon FA \rightarrow \int_D [\A(A,D), F^\dsharp D] \iso F^\dsharp A$,
  which is the unit, pseudo natural for $A$.

  The one side of triangular identity for the left adjoint ${(-)}^\dsharp$
  tells $\varepsilon_A\eta_A = \id$.
  Therefore, it suffices to show $\eta_A\varepsilon_A \iso \id$.
  To show this, we precompose $\widetilde{\lambda}_{CA}$ and postcompose
  $F^\dsharp A \iso \int_D [\A(D,A), F^\dsharp D] \xrightarrow{\rho_{DA}} [\A(D,A), F^\dsharp D]$
  to both $\eta_A\varepsilon_A$ and $\id$, where $\rho_{DA}$ is the universal wedge.
  By showing these are natural isomorphic,
  $\eta_A\varepsilon_A \iso \id$ is deduced from the universality of the pseudo coend and the end.

  Precompositition of $\widetilde{\lambda}_{CA}$ to $\varepsilon_A$ is $\overline{F_{AC}}$,
  where $\overline{F_{AC}}$ is the transpose of $F_{AC}\colon \A(A,C) \rightarrow [FC,FA]$,
  and postcomposition of $\rho_{DA}$ to $\eta_A$ is $\alpha_{AD}$.
  Thus, the left-hand side is $\alpha_{AD}\overline{F_{AC}}$.
  On the other hand, the right-hand side -- the composition of $\widetilde{\lambda}_{CA}$ and $\rho_{DA}$ -- is
  the transpose of
  \begin{align}
    \A(D,A) \times \A(A,C) \times FC \xrightarrow{c_{DAC} \times 1} \A(D,C) \times FC
    \xrightarrow{\widetilde{\lambda}_{CD}} F^\dsharp D.\label{eq:_righthandside}
  \end{align}
  Here, we used the 2-naturality of $\widetilde{\lambda}_{CA}$ for $A$.
  The pseudo naturality of $\widetilde{\lambda}_{CA}$ for $A$ produces the natural isomorphism,
  \[
    \begin{tikzpicture}
      \node(ul) at (0,2) {$\A(A,C)$};
      \node(dl) at (0,0) {$[\A(D,A)\times FC,\ \A(D,C)\times FC]$};
      \node(ur) at (6,2) {$[\A(D,A)\times FC,\ \A(D,A)\times FA]$};
      \node(dr) at (6,0) {$[\A(D,A)\times FC,\ F^\dsharp D].$};
      \node at (3,1) {$\iso$};
      \draw[1cell] (ul) to node {$\A(D,A)\times F-$} (ur);
      \draw[1cell,swap] (ul) to node {$\A(D,-)\times FC$} (dl);
      \draw[1cell] (ur) to node {$[1,\widetilde{\lambda}_{CA}]$} (dr);
      \draw[1cell] (dl) to node {$[1,\widetilde{\lambda}_{CD}]$} (dr);
    \end{tikzpicture}
  \]
  Carefully checked, it can be proved that $\alpha_{AD}\overline{F_{AC}}$ is the transpose of the upper right of the diagram,
  and \cref{eq:_righthandside} is the transpose of the down left.

  Showing the naturality of $A$ and $F$ requires a bit more work.
  That is because, as we showed in \Cref{prop:lax_naturality_for_each_variable},
  the naturality for each variable does not show the whole naturality.
  The compatibility of naturality with $C$ or $D$ needs to be checked,
  but we omit the proof here for redundancy.
  One can also check these naturality from 2-monad theory in the following section.
\end{proof}

In the same way, one can prove bicategorical Yonede lemma
$F^\dflat A \eqv FA$ dually.

\section{From the point of view of 2-monad theory}\label{section:2-monad_8}

The adjunction we showed at \cref{eq:adjunction_cat_dsharp} 
\begin{align}
  \Ps{\A^\op,\Cat} (F,H) & \iso
  [\A^\op, \Cat](F^\dsharp,H)     \\
  \Ps{\A^\op,\Cat} (H,F) & \iso
  [\A^\op, \Cat](H, F^\dflat)
\end{align}
can also be derived from 2-monad theory \cite{blackwell-2-monad}.
There is a 2-monad $T$ over the 2-category $[\ob(\A^\op), \Cat]$
whose strict algebras are 2-functors
and whose strict/pseudo/lax morphisms are 2-/pseudo/lax transformations.
This 2-category $[\ob(\A^\op, \Cat)]$ is complete and cocomplete, and 2-monad $T$ has a right adjoint \cite{Lack_Shulman}.
Therefore, $T$ satisfies the coherence condition \cite{Lack2005codescent}, that is,
\begin{itemize}
  \item The inclusion $T$--$\mathrm{Alg}_s = [\A^\op, \Cat] \hookrightarrow T$--$\mathrm{Alg} = \Ps{\A^\op, \Cat}$
  has a left adjoint ${(-)}^\dsharp$ and a right adjoint ${(-)}^\dflat$,
  \item whose units or counits in $\Ps{\A^\op, \Cat}$ are (pseudo) equivalences $F^\dsharp \eqv F \eqv F^\dflat$.
\end{itemize}
On the other hand, for lax natural transformations, there are two adjunctions in $\Lax{\A^\op, \Cat}$ between $F$ and $F^\sharp$, and between $F$ and $F^\flat$.

\end{document}